\documentclass[a4paper,12pt]{amsart} 
\title[Fields of definition of rational curves]{Fields of definition of rational curves of a given degree}

\usepackage{amsmath, amssymb, mathrsfs, amsthm, shorttoc, amsfonts}
\usepackage{mathtools} 
\usepackage{hyperref}
\usepackage[all]{xy}
\usepackage[left=3.5cm,right=3.5cm]{geometry}

\usepackage{enumitem}

\usepackage{cleveref}

\AtBeginDocument{\renewcommand{\ref}[1]{\cref{#1}}}

\usepackage{pgf,tikz}
\usetikzlibrary{matrix, calc, arrows}

\usepackage{tikz-cd} 

\newcommand{\on}[1]{\operatorname{#1}}
\newcommand{\bb}[1]{{\mathbb{#1}}}

\newcommand{\ca}[1]{{\mathcal{#1}}}
\newcommand{\bd}[1]{{\mathbf{#1}}}

\newcommand{\sub}{\subseteq}

\theoremstyle{definition}
\newtheorem{definition}{Definition}[section]

\theoremstyle{plain}
\newtheorem{conjecture}[definition]{Conjecture}
\newtheorem{proposition}[definition]{Proposition}
\newtheorem{lemma}[definition]{Lemma}
\newtheorem{theorem}[definition]{Theorem}
\newtheorem{corollary}[definition]{Corollary}

\theoremstyle{remark}

\newtheorem{remark}[definition]{Remark}

\renewcommand{\phi}{\varphi}

\author{David Holmes and Nick Rome}
\date{\today}

\address{Mathematisch Instituut\\ 
Universiteit Leiden\\
Postbus 9512\\
2300RA Leiden\\
Netherlands}

\email{holmesdst@math.leidenuniv.nl}

\address{School of Mathematics\\
University of Bristol \\ Bristol\\ BS8 1TW\\ UK}
\email{nick.rome@bristol.ac.uk}

\newcounter{nootje}
\newcommand{\beq}{\begin{equation}}
\newcommand{\eeq}{\end{equation}}
\newcommand{\beqs}{\begin{equation*}}
\newcommand{\eeqs}{\end{equation*}}

\DeclareMathOperator{\Mod}{mod}

\renewcommand{\d}{e}
\begin{document}
\maketitle


\begin{abstract} 
Kontsevich and Manin gave a formula for the number $N_\d$ of rational plane curves of degree ${\d}$ through $3{\d}-1$ points in general position in the plane. When these $3{\d}-1$ points have coordinates in the rational numbers, the corresponding set of $N_{\d}$ rational curves has a natural Galois-module structure. We make some extremely preliminary investigations into this Galois module structure, and relate this to the deck transformations of the generic fibre of the product of the evaluation maps on the moduli space of maps. 

We then study the asymptotics of the number of rational points on hypersurfaces of low degree, and use this to generalise our results by replacing the projective plane by such a hypersurface. 
\end{abstract}



\newcommand{\Mtildes}{ \widetilde{\ca M}^\Sigma}
\newcommand{\sch}[1]{\textcolor{blue}{#1}}

\newcommand{\Mbar}{\overline{\ca M}}
\newcommand{\MD}{\ca M^\blacklozenge}
\newcommand{\Md}{\ca M^\lozenge}
\newcommand{\DRL}{\operatorname{DRL}}
\newcommand{\DR}{\operatorname{DR}}
\newcommand{\isom}{\stackrel{\sim}{\longrightarrow}}
\newcommand{\Ann}[1]{\on{Ann}(#1)}
\newcommand{\fm}{\mathfrak m}

\section{Introduction}
For a given positive integer ${\d}$, we write $N_{\d}$ for the finite number of rational (i.e. geometric genus 0) plane curves of degree ${\d}$ through $3{\d}-1$ points in general position in $\bb P^2_{\bb C}$. We see immediately that $N_1 = N_2 = 1$, and the number $N_3$ of singular plane cubics through $8$ points is easily shown to be 12. Zeuthen \cite{ZeuthenAlmindelige-Ege} proved in 1873 that $N_4 = 620$, and Ran \cite{Ran1989Enumerative-geo} and Vainsencher \cite{Vainsencher1995Enumeration-of-} showed in the early 1990s that $N_5 = 87,304$. Around the same time Kontsevich and Manin \cite{Kontsevich1994Gromov-Witten-c} proved a general recursive formula 
\begin{equation*} 
N_{\d} = \sum_{\stackrel{{\d}_A + {\d}_B = {\d}}{{\d}_A >0, {\d}_B>0}}N_{{\d}_A}N_{{\d}_B}{\d}_A^2{\d}_B\left(  {\d}_B\left( \substack{3{\d}-4\\3{\d}_A - 2}\right) - {\d}_A  \left( \substack{3{\d}-4\\3{\d}_A - 1}\right) \right), 
\end{equation*}
allowing the rapid computation of any $N_{\d}$. Their proof ran via the intersection theory on moduli spaces of stable maps, and has initiated a vast area of research generalising this to curves of higher genera, and to more geometrically interesting targets in place of $\bb P^2$. 

In this paper we take a more arithmetic viewpoint, and ask over what sub-fields of $\bb C$ we should expect these rational curves to be defined. The most basic version of the question is the following: suppose the $3{\d}-1$ points in general position all have coordinates in the rational numbers $\bb Q$. Should we then expect any or all of the $N_{\d}$ rational curves to be defined over $\bb Q$ (i.e. such that their defining equations can be chosen to have rational coefficients, or equivalently to arise by base-change from some curves over $\bb Q$)? The answer is trivially `yes' if ${\d}=1$ or $2$ (since $N_1 = N_2 = 1$), but we will see later that the answer is `no' for all higher ${\d}$:

\begin{theorem}
Let $\d >2$. Then the set of $(3{\d}-1)$-tuples of points in $\bb P^2(\bb Q)$ where at least one of the $N_{\d}$ rational curves is defined over $\bb Q$ forms a thin set. 
\end{theorem}
The notion of a \emph{thin set} is due to Serre; in \ref{sec:thin} we will recall the definition and prove a stronger version of the above theorem. The idea of a thin set is that it should contain `few' points; indeed, from the above we easily deduce in \ref{sec:asymptotics}: 
\begin{corollary}
Ordering all $(3{\d}-1)$-tuples of points in $\bb P^2(\bb Q)$ by height, the proportion of tuples where at least one of the $N_\d$ curves is defined over $\bb Q$ is $0\%$. 
\end{corollary}


\newcommand{\f}{d}

Analogous questions can be asked about rational curves on hypersurfaces $X$ in $\bb P^N$ whose degree is low relative to their dimension. A lot is known about the irreducibility and dimension of the relevant moduli spaces, for generic $X$ by work of Harris--Roth--Starr \cite{Harris2004Rational-curves}, Beheshti--Kumar \cite{BeheshtiKumar} and Riedl--Yang \cite{RiedlYang}, and for any $X$ by work of Browning-Vishe \cite{BrowningVishe} and Browning--Sawin \cite{Browning2018Free-rational-c}. This allows us to prove analogously that very few $n$-tuples of points in $X$ are such that at least one element of the finite set of rational curves through them of suitable degree is defined over $\bb Q$.
\begin{theorem}
Let $X \sub \bb P^N$ be a smooth hypersurface of degree $\f$ such that ${({\f}-1)2^{{\f}}<N}$. Fix integers ${\d}>0$ and $n \ge 0$  such that the expression\\ ${(N + 1 - {\f})\d + (N - 4) + n = n(N-1)}$ holds.   
Then there exists a thin set $A \sub X^n(k)$ such that for all $P \in X^n(k) \setminus A$, the corresponding set of rational curves of degree $\d$ in $X$ through the points in $P$ contains no curve defined over $\bb Q$. 
\end{theorem}
Note that, unlike in the case of target $\bb P^2$, it is not always the case that we can fix the other variables and then find a value of $n$ which works. This is because we have to arrange that a certain product of evaluation maps is finite so that these sets of rational curves are finite, and this is not always possible for arbitrary choices of $N$, $\f$ and $\d$.

Since weak approximation is known to hold for such hypersurfaces \cite{Skinner}, it follows from \cite[theorem 3.5.7]{Serre2016Topics-in-Galoi} that the set of rational points is not thin. This means that, by the previous theorem, there must exist points $P \in X^n(k)$ for which none of the curves through $P$ are defined over $\mathbb{Q}$. In \ref{sec:thin points} we give a more refined quantitative estimate for the number of points which lie in any given thin set. Combining this with the above theorem yields:
\begin{corollary}
Under the above hypotheses, and assuming that $X(\bb Q)$ is non-empty, the proportion of $n$-tuples of points in $X(\bb Q)$ for which at least one of the rational curves is defined over $\bb Q$ is $0\%$. 
\end{corollary}
Informally, this means that `almost all' of the sets of rational curves do not contain any curve defined over $\bb Q$. In fact, \ref{thin theorem} is stronger than this, proving a power saving in the count for rational points in a thin set. The key tool in the proof is a sieve result for points on hypersurfaces which may be thought of as a form of effective strong approximation, and which the authors believe will also be of independent interest.

The questions we ask in this paper seem natural from an arithmetic perspective, but do they have interesting geometric content? As we will see in the proof of \ref{prop:not_transitive}, what we are really studying is the group of deck transformations of the generic fibre of the product of the evaluation maps on the moduli space of maps to $\bb P^2$ of given degree. Kontsevich and Manin established that this cover is of degree $N_{\d}$, but this leaves open the question of the structure of the group of deck transformations.

\subsection{Splitting fields}

Given a set $P$ of $3{\d}-1$ $\bb Q$-points of $\bb P^2$ in general position, write $C_P$ for the set of $N_{\d}$ (complex) rational plane curves of degree $\d$ through $P$.  We write $L_P \sub \bb C$ for the splitting field of the Galois module $C_P$, i.e. the smallest sub-field of $\bb C$ such that every curve in $C_P$ arises by base-change from some plane curve over $L_P$. Naively, we can think of $L_P$ as the field generated by the coefficients of defining equations for the curves, after scaling these equations to have at least one coefficient in $\bb Q$. 

The field extension $L_P/\bb Q$ is necessarily finite and Galois (in other words, the fixed field of the automorphism group of $L_P$ over $\bb Q$ is $\bb Q$ itself). We can ask about the degree of $L_P$ over $\bb Q$ or (for a finer invariant) its Galois group $\on{Gal}(L_P/\bb Q)$. At one extreme we could have $L_P  = \bb Q$, in other words all $C_P$ are defined over $\bb Q$. At the other extreme it could have $L_P$ of degree $N_{\d}!$ and Galois group the symmetric group $S_{N_{\d}}$ on $N_{\d}$ objects. We conjecture that the latter occurs `almost always'. More precisely, we propose:
\begin{conjecture}\label{conj:S_N_d}
For all $P$ outside a thin subset of $(\bb P^2)^{3{\d}-1}(\bb Q)$, we have $\on{Gal}(L_P/\bb Q) = S_{N_{\d}}$. 
\end{conjecture}
See \ref{sec:thin} for the definition of a thin set. The conjecture is trivially true for ${\d}=1$ and $2$, and we prove it in \ref{sec:d_3} for ${\d}=3$. 

%
%
%
%


\subsection{Further questions}

There are a number of possible variations and extensions on the questions proposed in this note. One can ask whether the 0-dimensional schemes $C_p$ satisfy the Hasse Principle, or whether there is a Brauer-Manin obstruction. 

It also seems to be interesting to understand what happens in positive characteristic, but to the authors' knowledge even the number of rational curves has not been determined over $\bar{\bb F}_p$ (it is clear that the number coincides with that over $\bb C$ for `large enough' $p$, but making this `large enough' explicit, and understanding what happens for small $p$, seems to remain open). 

It may well be possible to extend the computations in \ref{sec:d_3} to ${\d}=4$ or maybe even ${\d}=5$, though the authors do not have the courage to attempt it. It is clear that other techniques will be needed in the general case. 

\subsection{Atttribution}
Sections 1 - 4 (excluding 2.1) are due to the first-named author. Sections 5 and 2.1 are due to the second-named author. 
\subsection{Acknowledgements}
Both authors are grateful to Tim Browning for putting them in contact with one another, and for helpful comments. 
\subsection{Notation}
We shall write $\bb Z^n_{\text{prim}}$ to denote the set of primitive vectors in $\bb Z^n$. A sum with subscript ``dyadic" will refer to a sum whose variables run over powers of 2. As is standard, we will write $f(x) \ll g(x)$ to mean that there exists some constant $c>0$ such that for all sufficiently large $x$, we have $\vert f(x) \vert \leq c \vert g(x)\vert$, and $f(x) \asymp g(x)$ to mean $f(x) \ll g(x) \ll f(x)$. 

\section{Most sets of $N_{\d}$ curves contain no curve defined over the ground field}
\label{sec:thin}
Let $k$ be a field of characteristic zero, and $X/k$ a reduced separated scheme of finite type (a \emph{variety}). For $X$ an irreducible variety, a subset $A \sub X(k)$ is called \emph{thin} if there exists a map of varieties $\pi\colon Y \to X$, not admitting a rational section, and such that $A \sub \pi(Y(k))$. The field $k$ is \emph{Hilbertian} if $\bb P^1(k)$ is not thin (as a subset of itself). The field $\bb Q$ is Hilbertian, as is any finitely generated extension; $\bb C$ is not. 

\begin{remark}
If there exists a closed subset $Z \subset X$ with $Z \neq X$ and $A \subset Z(k)$ then we refer to $A$ as a thin set of type I.
If there is some irreducible variety $Y$ with $\text{dim}(X) = \text{dim}(Y)$ and a generically surjective morphism $\pi: Y \rightarrow X$ of degree $\geq 2$ with $A \subset \pi(Y(k))$ then $A$ is referred to as a thin set of type II.
Any thin subset $A \subset X(k)$ may be written as a finite union of thin sets of type I and type II.
\end{remark}

Let $k \sub \bb C$ be a Hilbertian field, and fix a positive integer ${\d}$. Suppose we are given $P \in (\bb P^2)^{3{\d}-1}(k)$; we can think of this as a $(3{\d}-1)$-tuple of $k$-points in $\bb P^2$, and we write $C_P$ for the set of rational curves of degree ${\d}$ through $P$ defined over $\bar k$; in general this set may be infinite. We say $P$ is \emph{transitive} if 
\begin{itemize}
\item
the set $C_P$ has $N_{\d}$ elements (this holds if $P$ is in `general position'), and 
\item 
the Galois group $\on{Gal}(\bar k /k)$ acts transitively on the set $C_P$ (equivalently, the Galois group of the splitting field $L_P$ of $C_P$ is a transitive subgroup of the symmetric group on $N_{\d}$ elements). 
\end{itemize}
Note that $P$ being transitive implies in particular that none of the curves in $C_P$ is defined over $k$, for ${\d} >2$. 
\begin{proposition}\label{prop:not_transitive}
The set of $P \in (\bb P^2)^{3{\d}-1}(k)$ which are not transitive form a thin set. 
\end{proposition}
\begin{proof}We may and do assume that ${\d} >2$, otherwise the result is obvious. The result is immediate from a lemma of Serre, after the standard reformulation of the counting problem into a moduli problem. We write $\ca M_{0,3{\d}-1}(\bb P^2, {\d})$ for the moduli space whose $T$-points are tuples $(C/T, x_1, \dots, x_{3{\d}-1}, \phi)$ with $C/T$ a smooth curve of genus 0, the $x_i \in C(T)$ disjoint sections, and $\phi\colon C \to \bb P^2$ is fibrewise generically immersive and satisfies $\phi^*\ca O_{\bb P^2}(1) \cong \ca O_{\bb P^1}({\d})$. It is easy to check that $\ca M_{0,3{\d}-1}(\bb P^2, {\d})$ is an \emph{irreducible} variety. 

This space $\ca M_{0,3{\d}-1}(\bb P^2, {\d})$ comes with evaluation maps $ev_i$ to $\bb P^2$ for $1 \le i \le 3{\d}-1$, sending $(C/T, x_1, \dots, x_{3{\d}-1}, \phi)$ to $\phi\circ x_i$. Together these $ev_i$ induce a map $ev\colon \ca M_{0,3{\d}-1}(\bb P^2, {\d}) \to (\bb P^2)^{3{\d}-1}$, generically finite of degree $N_{\d}$, and $C_P = ev^{-1}(P)$. 

Denote by $\eta$ the generic point of $(\bb P^2)^{3{\d}-1}$, then the fibre $ev^{-1}\eta$ is irreducible since $\ca M_{0,3{\d}-1}(\bb P^2, {\d})$ is. Then by \cite[prop. 3.3.5]{Serre2016Topics-in-Galoi} 
there exists a thin set $A \sub (\bb P^2)^{3{\d}-1}(k)$ such that for all $P$ outside $A$, the fibre $ev^{-1}P$ has the expected number $N_{\d}$ of $\bar k$-points and is irreducible, so the Galois action is transitive. 
\end{proof}
In fact, the same proof shows more: we can choose the thin set $A$ such that for every $P$ outside $A$, the Galois group of $L_P$ is naturally isomorphic to that of the generic fibre $ev^{-1}\eta / \eta$. So to prove \ref{conj:S_N_d} it would be equivalent to show that the Galois group of the generic fibre $ev^{-1}\eta$ were the full symmetric group $S_{N_{\d}}$. 


\subsection{Asymptotics}\label{sec:asymptotics}
In this subsection we take $k = \bb Q$. There are a number of senses in which thin sets contain `few' points. One of them is by counting the number of points up to a given size. For a positive integer $B$ we define $M(B)$ to be the number of points in $(\bb P^2)^{3{\d}-1}(\bb Q)$ with height bounded by $B$. We choose to use here an anticanonical height on $(\bb P^2)^{3e-1}$.
If $x=(x_0: x_1:x_2) \in \bb P^2$ and $\mathbf{x} =(x_0, x_1, x_2) \in  \bb Z_{\text{prim}}^3$, then the height of $x$ associated to the anticanonical bundle $\mathcal{O}(3)$ is given by
\[ H(x) = \vert \vert \mathbf{x} \vert \vert_{\infty}^{3},\] where $\vert \vert \mathbf{x} \vert \vert_{\infty} = \max \limits_i \vert x_i \vert.$
The height on $(\bb P^2)^{3e-1}$ is then inherited from this since the anticanonical bundle on the product is $\mathcal{O}(3,3, \dots, 3)$ and hence an anticanonical height on $(\bb P^2)^{3e-1}$ is given by a product of the heights on each copy of $\bb P^2$.
If $x \in (\bb P^2)^{3{\d}-1}(\bb Q)$ there is a representative of $x$ of the form $(\mathbf{x}^{(1)}, \dots, \mathbf{x}^{(3e-1)}) \in(\bb Z_{\text{prim}}^3)^{3e-1}$, the height of $x$ is given by 
$$H(x) = \prod_{j=1}^{3e-1} \vert \vert \mathbf{x}^{(j)} \vert \vert_{\infty}^3.$$
By the compatibility of Manin's conjecture with taking products (initially observed in \cite[section 1]{FMT}), we have $M(B) \asymp B(\log B)^{3e-2}$.

We define $N(B)$ to be the number of points $P$ of height at most $B$ which are {\em not} transitive.

\begin{lemma}\label{P2 count}
There exists $\gamma \in [0,1)$ such that
\[
N(B) = O\left( B (\log B)^{3e-3+\gamma} \right).
\]
\end{lemma}

\begin{proof}
Let $A \subset (\bb P^2)^{3{\d}-1}(\bb Q)$ be any non-empty thin set. We will prove that the number of points of height at most $B$ which lie in $A$ can be bounded above by $O\left(B(\log B)^{3e-3+\gamma}\right)$, from which (along with \Cref{prop:not_transitive}) the result follows. Note that it suffices to just consider the case when $A$ is a thin set of type I or of type II. 

We count points in $(\bb P^2)^{3e-1}(\bb Q)$ by passing to the affine cone $(\bb Z_{\text{prim}}^3)^{3e-1}$. The affine cone of $A$ is
\[
A' = \{ x \in (\bb Z_{\text{prim}}^3)^{3e-1} : (\mathbf{x}^{(1)}: \ldots: \mathbf{x}^{(3e-1)}) \in A\}
\]
 and we denote by $\overline{A}_p$ the reduction of $A'$ modulo a prime $p$.
We will upper bound the number of elements in $A'$ of bounded height (and thus those in $A$) by the cardinality
\[
S(B) := \#\left\{ x \in (\bb Z^3_{\text{prim}})^{3{\d}-1} : \prod_{j=1}^{3e-1} \vert \vert \mathbf{x}^{(j)}\vert\vert_{\infty} \leq B^{1/3} \text{ and } x \Mod p \in \overline{A}_p\right\}.
\]
To attack this we first break into dyadic intervals
\[
S(B) \leq 
\sum_{\substack{B_1\cdots B_{3e-1} \leq B^{1/3}\\ \text{dyadic}}} 
H(B_1, \dots, B_{3e-1}),
\]
where $H(B_1, \dots, B_{3e-1})$ is defined to be 
\[
\#\{ x \in (\bb Z_{\text{prim}}^3)^{3{\d}-1} : B_j/2<\vert \vert \mathbf{x}^{(j)} \vert \vert_{\infty} \leq B_j\, \forall j,\text{ and } x \Mod p \in \overline{A}_p\}.
\]
The inner cardinality can be estimated using the multi-dimensional large sieve in lopsided boxes (see e.g.\ \cite[theorem 4.1]{Kow}).
This gives
\[
S(B) \ll \sum_{\substack{B_1\cdots B_{3e-1} \leq B^{1/3}\\ \text{dyadic}}} 
  \left(\prod_{j=1}^{3e-1} \left(B_j + Q^2\right)^3\right) /G(Q), 
\]
where 
\begin{align*} G(Q) &:= \sum_{q \leq Q} \mu^2(q) \prod_{p \mid q} \frac{\omega_p}{1- \omega_p},
\\ \omega_p &:= 1 - \frac{\#\overline{A}_p}{\#(\bb F_p^3 \setminus\{\boldsymbol{0}\})^{3e-1}}.\end{align*}
If $A$ is a thin set of type II, then by \cite[thm. 3.6.2]{Serre2016Topics-in-Galoi}, there exists $c \in (0,1)$ and a finite Galois extension $k/ \bb Q$ such that for all sufficiently large primes $p$ which split completely in $k$, we have
 $$\#\overline{A}_p\leq c p^{3(3e-1)} + O(p^{3(3e-1) - 1/2}).$$ 
We denote the set of such sufficiently large, completely splitting primes by $\mathcal{P}$ and let $\delta$ be the natural density of such primes (which is strictly greater than 0 by Chebotarev's density theorem).
Therefore there exists $\eta < (1-c)/c$ such that for all primes $p \in \mathcal{P}$, we have $\frac{\omega_p}{1- \omega_p} \geq \eta$. Thus $$G(Q) \gg \sum_{\substack{q \leq Q\\ p \mid q \implies p \in \mathcal{P}}} \mu^2(q) \eta^{\omega(q)}.$$
This sum is estimated using Wirsing's theorem \cite[satz 1.1]{Wirsing}. The Chebotarev density theorem, along with an application of partial summation, tells us that
\[
\sum_{\substack{p \leq Q \\ p \in \mathcal{P}}} \frac{\eta \log p}{p} \sim \eta \delta \log Q.
\] 
Therefore by \cite[satz 1.1]{Wirsing} we have
\[
\sum_{\substack{q \leq Q\\ p \mid q \implies p \in \mathcal{P}}} \mu^2(q) \eta^{\omega(q)}
\sim
\frac{Q}{\log Q} \prod_{\substack{p \leq Q \\ p \in \mathcal{P}}}\left( 1 + \frac{\eta}{p} \right).
\]
By taking the logarithm of the above product, we have
\[
\log \left( \prod_{\substack{ p \leq Q \\ p \in \mathcal{P}}} \left(1 + \frac{ \eta }{p}\right) \right)
=
\sum_{\substack{ p \leq Q \\ p \in \mathcal{P}}} \log \left( 1+ \frac{\eta}{p} \right)
\gg
\delta \eta \log \log Q,
\] where the inequality is again a straightforward consequence of the Chebotarev density theorem.
Hence, we conclude that
\[
G(Q) \gg Q (\log Q)^{\delta \eta -1}.
\]
When $A$ is a thin set of type I, the Lang--Weil estimate (e.g.\ \cite[thm. 3.6.1]{Serre2016Topics-in-Galoi}) tells us that
\[
\#\overline{A}_p \ll p^{3(3e-1) -1},
\]
from which the bound $G(Q) \gg_{\epsilon} Q^{2-\epsilon}$ can be deduced in a similar manner.
Let $\gamma = 1-\delta \eta$, then setting $Q = \min \limits_j\{B_j^{1/2}\}$ we get
\begin{align*}
S(B) &\ll (\log B)^{\gamma}
\sum_{\substack{B_1\cdots B_{3e-1} \leq B^{1/3}\\ \text{dyadic}}} 
(B_1 \cdots B_{3e-1})^3 \min \limits_j\{B_j^{-1/2}\}\\
&\ll (\log B)^{\gamma} \sum_{\substack{B_1 \leq B_2 \leq \dots \leq B_{3e-1}\\ B_1 \cdots B_{3e-1} \leq B^{1/3}\\ \text{dyadic}}}
(B_1 \cdots B_{3e-1})^3 B_1^{-1/2}\\
&\ll B(\log B)^{\gamma}\sum_{\substack{B_1 \leq B_2 \leq \dots \leq B_{3e-2}\\ B_1 \cdots B_{3e-2} \leq B^{1/3}\\ \text{dyadic}}}
 B_1^{-1/2}\\
&\ll B(\log B)^{3e-3 + \gamma} \sum_{\substack{B_1 \leq B^{1/3}\\ B_1 = 2^j \text{ for some } j}}B_1^{-1/2}
.
\end{align*}
This final sum is convergent and so we deduce the claimed bound.
\end{proof}
In other words, the ratio $N(B)/M(B)$ tends to zero at least as fast as $(\log B)^{\gamma -1}$ as $B \to \infty$.

\section{Splitting fields for ${\d}=3$}\label{sec:d_3}
We know that $N_3 = 12$, and by \ref{conj:S_N_d} we should expect that the Galois group is $S_{12}$ for all $P$ outside some thin set; here we verify that. By the same argument as in the proof of \ref{prop:not_transitive} it is enough to verify that the Galois group of the splitting field of the generic fibre $L_\eta = ev^{-1}\eta$ is $S_{12}$. If $P$ is any point in $(\bb P^2)^{8}$ with $\# C_P = 12 $ (so $ev$ is \'etale in a neighbourhood of $P$) then we have a natural injection from the Galois group of $L_P$ to the Galois group of $L_\eta$. Hence it suffices to find a single such $P$ for which we can show that the Galois group of $L_P$ is $S_{12}$. 

For readability we will describe our example in affine coordinates, on one of the standard charts of $\bb P^2$. A little random experimentation brought us to the 8 points $(1,0), (-3,1), (3,-3), (-3,3), (2,-1), (-3,1), (2,4), (1,3)$ (in the first affine patch). The 2-dimensional space of cubics through these eight points is spanned by the cubics
\begin{equation*}
x^3 + 5/74y^3 + 28/37x^2 -10/37xy -50/37y^2 -173/37x + 275/74y + 108/37
\end{equation*}
and 
\begin{equation*}
x^2y -75/74y^3 + 61/37x^2 + 39/37xy + 10/37y^2 + 227/37x + 463/74y - 288/37. 
\end{equation*}
We used SAGE to re-write the generic element of the linear span of these two cubics into Weierstrass form, and to compute the discriminant of the resulting cubic, given by 
\begin{equation*}
\begin{split}
\Delta = -\frac{24953575063474272765882500}{6582952005840035281}t^{12} & + \frac{134142884940172812137734125}{6582952005840035281}t^{11} \\
+ \frac{1421147382955123926407537625}{26331808023360141124}t^{10} & - \frac{7215115079178977822483091875}{26331808023360141124}t^9 \\
+ \frac{17943310292096345174858538375}{105327232093440564496}t^8 & +\! \frac{16439749600766827372335403125}{52663616046720282248}t^7 \\
- \frac{3707387509942909790535756625}{13165904011680070562}t^6 & - \frac{1591463047251983769286633125}{26331808023360141124}t^5 \\ - \frac{400455774447542072616372750}{6582952005840035281}t^4 &- \frac{1770912522449250247984242375}{52663616046720282248}t^3 \\ - \frac{41930948369778731917638750}{6582952005840035281}t^2 & - \frac{12906646230435281098756875}{26331808023360141124}t \\- \frac{866622835858183959391875}{105327232093440564496}. 
\end{split}
\end{equation*}
A curve of degree 3 is rational if and only if it is not smooth, and the vanishing of the discriminant detects exactly when this non-smoothness occurs. In other words, the scheme of zeros of $\Delta$ is isomorphic to the scheme $C_P = ev^{-1}P$, so the splitting field $L_P$ is given by the splitting field of $\Delta$. We computed the Galois group of the latter in {\tt MAGMA}, and found it to be $S_{12}$ as required. 

\begin{remark}
On can alternatively argue using Del Pezzo surfaces (c.f. \cite[\S 7.3]{Varilly-Alvarado-Zywina}), though in the end this simply re-phrases the computer search. The blow up of $\bb P^2$ in 8 points in general position is a del Pezzo surface $X$ of degree 1, and the strict transforms of the 12 singular cubics give the 12 singular elements of the anticanonical linear system on $X$. Blowing up the 9th point through which they all pass yields an elliptic surface, and (the strict transforms of) these 12 curves are the singular fibres of the elliptic fibration. Explicitly, the surface $X$ has the equation:
$$y^2 = x^3 + f_2(z,w)x^2 + f_4(z,w)x + f_6(z,w)$$
where $f_i$ has degree $i$. The 12 points over which the fibres are non-smooth are the zero locus of the discriminant of this elliptic pencil, which is a polynomial $D(z,w)$ of degree 12 in $z$, $w$. One can then perform a computer search for polynomials $f_i$ for which $D(z,w)$ has Galois group $S_{12}$. 
\end{remark}





\section{Hypersurfaces of low degree}
\subsection{Formalities}
We continue to work over a Hilbertian field $k$ with algebraic closure $\bar k$. The reader will note that the only properties of the variety $\bb P^2$ used in the proof of \ref{prop:not_transitive} are the following:
\begin{enumerate}
\item
$\ca M_{0,3{\d}-1}(\bb P^2;{\d})$ is an irreducible variety;
\item the product of the evaluation maps $ev\colon \ca M_{0,3{\d}-1}(\bb P^2;{\d}) \to (\bb P^2)^{3{\d}-1}$ is generically finite. 
\end{enumerate}

This naturally leads us to a generalisation of \ref{prop:not_transitive}. Let $X/k$ be an irreducible variety, fix a line bundle\footnote{More generally fix a class in \'etale cohomology of suitable dimension.} $\ca L$ on $X$, and for each non-negative integer $n$ define a moduli functor $\ca M_{0,n}(X, \ca L;{\d})$ on the category $Sch_k$ of schemes over $k$, sending $T$ to the set of tuples $(C/T, x_1, \dots, x_n, \phi)$ where
\begin{itemize}
\item
$C/T$ is a smooth proper curve of genus 0;
\item
The $x_i \in C(T)$ are disjoint sections;
\item
$\phi\colon C \to X_T$ is a $T$-morphism which is generically immersive on each fibre, such that $\phi^*\ca L$ has degree ${\d}$ on each fibre of $C/T$, 
\end{itemize}
modulo isomorphisms over $X$ (such isomorphisms are unique when they exist, so this is a fine moduli space). This functor is not always an irreducible variety (for example, if $\ca L = \ca O_X$ and ${\d}=0$ it is likely to have infinitely many connected components; the quintic 3-fold provides a much less trivial example), but sometimes it is an irreducible variety, for example when $X$ is $\bb P^2$ and $\ca L = \ca O(1)$.


As before we have evaluation maps $ev_i\colon \ca M_{0,n}(X, \ca L;{\d}) \to X$ sending a tuple $(C/T, x_1, \dots, x_n, \phi)$ to $\phi \circ x_i$, and can take their product $ev\colon \ca M_{0,n}(X, \ca L;{\d}) \to X^n$. The set of $\bar k$-points of the fibre of $ev$ over a point $P = (p_1, \dots, p_n) \in X^n(k)$ is exactly the set of rational curves in $X$ over $\bar k$, of $\ca L$-degree ${\d}$, and passing through all the $p_i$. In general this set can be infinite, but if $ev$ is generically finite (i.e. the fibre of $ev$ over the generic point of $X^n$ is finite) then, for $P$ outside some proper Zariski closed subset of $X^n$, these sets are finite and all of the same cardinality, which we shall denote $N(X,\ca L, n, {\d})$. In this case we say $P$ is \emph{transitive} if indeed the fibre has the `expected' number $N(X,\ca L, n, {\d})$ of $\bar k$-points, and moreover the natural action of $\on{Gal}(\bar k/k)$ on the fibre is transitive.  Imitating the proof of \ref{prop:not_transitive} one immediately obtains
\begin{proposition}\label{prop:general_non_transitive}
Fix a variety $X$, a line bundle $\ca L$ on $X$, and non-negative integers ${\d}$ and $n$. Assume that 
\begin{enumerate}
\item
$\ca M_{0,n}(X, \ca L;{\d})$ is an irreducible variety;
\item the product of the evaluation maps $ev\colon \ca M_{0,n}(X, \ca L;{\d}) \to X^n$ is generically finite. 
\end{enumerate}
Then there exists a thin subset $A \sub X^n(k)$ such that all $P \in X^n(k) \setminus A$ are transitive. 
\end{proposition}

Note that we do not exclude the possibility that the generic fibre of $ev$ is empty, but any map from the empty scheme is finite, and for these purposes we consider the unique group action on the empty set to be transitive, making the result vacuous in this case. 

Connoisseurs of the empty set will consider this in poor taste --- `transitive' should morally mean `has exactly one orbit', so that the action on the empty set should not be considered transitive. Such readers should add to our result the assumption that the generic fibre of the map $ev$ is non-empty; understanding when this happens is a very interesting problem. 

Suppose that $ \ca M_{0,n}(X, \ca L;{\d})$ is an irreducible variety (in particular reduced); if the tangent map to $ev$ is surjective at some point of the source, it follows that  $ev$ is dominant (and the converse holds in characteristic zero, by generic smoothness). Surjectivity of the tangent map can be analysed via deformation theory; details can be found in \cite[\S 4]{DebarreHDAG}. For example, one can show (combining results of \cite{Shen_normal_bundle} and \cite{CoskunRiedl}) that, for a general cubic 3-fold the product of evaluation maps $ev\colon \ca M_{0,n}(X, \ca L;{\d}) \to X^n$ is dominant whenever $n \le e-1$ and $e \ge 2$. Note however that in this setting the dimensions are not equal (see below for a more detailed analysis), so that we do not get a generically finite map.

\subsection{Hypersurfaces}

For an interesting application of \ref{prop:general_non_transitive} we need an irreducible variety $X$ with two properties. First, it should satisfy the criteria of \ref{prop:general_non_transitive}. But we should also ask that $X(k)$ is itself not thin, otherwise the conclusion is vacuous. Such examples are provided by hypersurfaces in projective space of low degree (relative to their dimension).

Fix positive integers $N$, ${\f}$ and ${\d}$, and let $X$ be a smooth hypersurface of degree ${\f}$ in $\bb P^N$; we fix $\ca L = \ca O(1)$, and drop it from the notation henceforth. When is $\ca M_{0,n}(X;{\d})$ an irreducible variety? Note that $\ca M_{0,n}(X;{\d})$ can be built from $\ca M_{0,0}(X;{\d})$ by repeatedly taking universal curves (and deleting loci where sections intersect), so $\ca M_{0,n}(X;{\d})$ is an irreducible variety if and only if $\ca M_{0,0}(X;{\d})$ is, and their dimensions differ by $n$. There are two main cases when $\ca M_{0,0}(X;{\d})$ is known to be an irreducible variety:
\begin{enumerate}
\item
$\on{char} k = 0$, $N>2$ and $\f+2<N$, and $X$ is generic, by work of Reidl and Yang \cite{RiedlYang};
\item 
$N>2$ and $(2{\f}-1)2^{{\f}-1}<N$ by work of Browning and Sawin \cite{Browning2018Free-rational-c}, with no genericity assumptions on $X$, but assuming $\on{char}k = 0$ or $\on{char} k >\f$. 
\end{enumerate}
In these cases it is also known that $\ca M_{0,0}(X;{\d})$ is of the `expected' dimension $(N + 1 - {\f}){\d} + (N - 4)$. Since the dimension of $X^n$ is $n(N-1)$, the map $ev\colon \ca M_{0,n}(X, \ca L;{\d}) \to X^n$ is generically finite if and only if we have the equality ${(N + 1 - {\f}){\d} + (N - 4) + n = n(N-1)}$ (it could perhaps happen that the map $ev$ is not dominant, but in this case the generic fibre of $ev$ is empty, so the result is vacuously true as remarked above). We deduce:
\begin{proposition}\label{prop_hypersurfaces_thin}
Let $N>2$, ${\f}>0$ and ${\d}>0$ and $n \ge 0$ be integers such that $(N + 1 - {\f})\d + (N - 4) + n = n(N-1)$. Let $X$ be a hypersurface in $\bb P^N$ of degree ${\f}$, satisfying one of the assumptions (1) and (2) above. 
Then there exists a thin set $A \sub X^n(k)$ such that all $P \in X^n(k) \setminus A$ are transitive. 
\end{proposition}

\subsection{Asymptotics}
We keep the notation of the above section, but restrict now to the case $k = \bb Q$. A-priori it could be that $A = X^n(\bb Q)$, but using the results from the appendix we can show this is far from the case. First, we must assume that $X$ has at least one rational point (by Birch's result below, this is equivalent to assuming that $X$ is everywhere locally soluble). 

Then for a positive integer $B$, write $\bd M(B)$ for the number of rational points on $X^n$ of height less than $B$. Assuming that $N > 2^\f({\f}-1) -1$, Birch \cite[theorem 1]{Birch}, combined with the compatibility of Manin's conjecture with products as noted above, shows that there exist $c>0$ and $\delta >0 $ such that 
\begin{equation*}\label{Eq:birch}
\bd M(B) = c B (\log B)^{n-1}+ O\left( B(\log B)^{n-1 - \delta}\right). 
\end{equation*}

Now write $\bd N(B)$ for the number of $P \in X^n(\bb Q)$ which are not transitive. Assume $N>2^{\f}(2{\f}-1) $, then by \ref{prop_hypersurfaces_thin} of the preceding section, there exists a thin subset $A \sub X^n$ containing all those $P$ which are not transitive. Now we can apply \ref{thin theorem} to see that there exists $\gamma \in [0,1)$ with
\begin{equation*}
\bd N(B)  \ll B (\log B)^{n-2+\gamma}
\end{equation*}

Hence the ratio
\begin{equation*}
\frac{\bd N(B)}{\bd M(B)}  \ll \frac{ B (\log B)^{n-2+\gamma}}{  c B (\log B)^{n-1}+  c'B(\log B)^{n-1 - \delta}} \ll  (\log B)^{\gamma-1}
\end{equation*}
where $c'$ is some positive constant. Summarising, we have
\begin{theorem}
Let $X/\bb Q$ be a hypersurface with $N>2^{\f}({\f}-1)$. Then there exists $\delta' >0$ such that 
\begin{equation*}
\frac{\#\{P \in X^n(\bb Q): H(P) \le B \text{ and } P \text{ not transitive}\}}{\#\{P \in X^n(\bb Q): H(P) \le B\}} \ll (\log B)^{-\delta'}. 
\end{equation*}
\end{theorem}

 Informally, this tells us that most collections of points on a hypersurface $X$ are transitive, thus the sets of curves through them contain no curves defined over $\bb Q$.

\section{Thin Sets on Smooth Hypersurfaces of Low Degree}
\label{sec:thin points}


\newcommand{\PP}{\mathbb{P}}
\renewcommand{\AA}{\mathbb{A}}
\newcommand{\A}{\mathbf{A}}
\newcommand{\FF}{\mathbb{F}}
\newcommand{\ZZ}{\mathbb{Z}}
\newcommand{\GG}{\mathbb{G}}
\newcommand{\NN}{\mathbb{N}}
\newcommand{\QQ}{\mathbb{Q}}
\newcommand{\RR}{\mathbb{R}}
\newcommand{\CC}{\mathbb{C}}
\newcommand{\Zprim}{\ZZ_{\text{prim}}}

\newcommand{\h}{\mathbf{h}}
\renewcommand{\d}{\mathbf{d}}
\newcommand{\s}{\mathbf{s}}
\newcommand{\m}{\mathbf{m}}
\newcommand{\n}{\mathbf{n}}
\newcommand{\M}{\mathbf{M}}
\newcommand{\x}{\mathbf{x}}
\newcommand{\y}{\mathbf{y}}
\renewcommand{\c}{\mathbf{c}}
\renewcommand{\v}{\mathbf{v}}
\renewcommand{\u}{\mathbf{u}}
\newcommand{\z}{\mathbf{z}}
\renewcommand{\b}{\mathbf{b}}
\renewcommand{\a}{\mathbf{a}}
\renewcommand{\k}{\mathbf{k}}
\newcommand{\w}{\mathbf{w}}

\renewcommand{\l}{\left}
\renewcommand{\r}{\right}

\newcommand{\norm}[1]{\left \vert \left \vert {#1} \right \vert \right \vert}

Let $F \in \ZZ[x_0, \dots, x_N]$ be a homogeneous form of degree $d$. Denote by $X$ the projective variety defined by $F$ and assume (for simplicity) that it is smooth. Given $n \in \NN$, suppose $A \subset X^n(\QQ)$ is a non-empty thin set. The purpose of this section is to show that a thin set on $X^n$ contains few points, extending a result of Browning--Loughran \cite[Theorem 1.8]{BL} on the number of points in a thin subset of a quadric to multiple copies of a general hypersurface (of suitably large dimension). 

\begin{theorem}\label{thin theorem}
If $N> 2^d (d-1) $ and $H$ is the anticanonical height function on $X^n$ described below then $\exists \,\delta, \gamma>0$ such that
\[
\# \{ x \in A : H(x) \leq B \} \ll_{A, X}
\left\{
\begin{array}{ll}
B^{1-\delta}(\log B)^{\gamma} &\text{ if } n=1,\\
B(\log B)^{n-2 + \gamma}&\text{ if } n\geq2. 
\end{array}\right.
\]
\end{theorem}
\begin{remark}
Let $x \in X^n(\QQ)$ and $x = (\mathbf{x}^{(1)}, \dots, \mathbf{x}^{(n)})$ for $\mathbf{x}^{(j)} \in \ZZ_{\text{prim}}^{N+1}$.
Since the anticanonical bundle on $X$ is $\mathcal{O}(N+1-d)$ (see e.g \cite[section 1]{FMT}), an anticanonical height function on $X^n$ is given by 
$$H(x) = \prod_{j=1}^n \vert \vert \mathbf{x}^{(j)} \vert \vert _{\infty}^{N+1 - d}.$$
\end{remark}
The result is a consequence of a sieve estimate for points on products of a hypersurface lying in some prescribed residue classes. In establishing this estimate we make crucial use of a recent generalisation of Birch's theorem due to Schindler--Sofos \cite{SS} (c.f. \ref{lopsided birch}).

Fix $m\in \NN$. Let 
\begin{equation*}\label{eq:classes}
\Omega_{p^m}  \subset 
\{(\x^{(1)}, \dots, \x^{(n)}) \in \l(\ZZ/p^m \ZZ \r)^{n(N+1)} : p \nmid \x^{(j)}, F(\x^{(j)}) \equiv 0 \Mod p^m \, \forall j\}
\end{equation*}
be some non-empty collection of residue classes for each prime $p$. Denote their relative density by
\[
\omega_p := 1 - \frac{\#\Omega_{p^m}}{\#\widehat{X^n}\l(\ZZ/p^m\ZZ\r)} \in [0,1),
\] where $\widehat{X^n}$ denotes the affine cone of $X^n$.
We will establish \ref{thin theorem} using a large sieve type estimate for points of bounded height of the following form.
\begin{lemma}
There exist $\delta_1, \delta_2 >0$ such that
\[
\#\{ x \in X(\QQ) : H(x) \leq B \text{ and } x \Mod p^m \in \Omega_{p^m} \forall p\}
\ll
\frac{B}{\min\{G(B^{\delta_1}),B^{\delta_2}\}},
\]
where
\[
G(Q) = \sum_{q \leq Q} \mu^2(q) \prod_{p \mid q} \frac{\omega_p}{1- \omega_p}.
\]
\end{lemma}

\begin{remark}
This theorem is analogous to \cite[theorem 1.7]{BL} and all the results derived for quadrics concerning fibrations, zero loci of Brauer group elements and friable divisors could also be generalised to the setting of smooth hypersurfaces of low degree in a similar fashion.
\end{remark}
In general we will need to look at products of hypersurfaces, to deal with the resulting height condition we break into dyadic intervals as in the proof of \ref{P2 count}.
Then we need to investigate the subset $\mathbf{N}(X^n,\mathbf{B},\Omega)$ defined by
\[
\#\{
x \in \widehat{X^n}(\QQ): B_{j}/2 < \vert \vert \x^{(j)} \vert \vert_{\infty} \leq B_j \text{ and } \x^{(j)} \Mod p^m \in \Omega_{p^m} \forall p,j
\},
\]
where $(\x^{(1)}, \dots, \x^{(n)})$ is a representative of $x$ in $(\ZZ_{\text{prim}}^{N+1})^n$.
We have the following analogous estimate for products.
\begin{theorem}\label{sieve theorem} Suppose $N> 2^d(d-1) $.
Then for any $\epsilon >0$ and any $Q \geq 1$, one has
\[
\mathbf{N}(X^n,\mathbf{B},\Omega) \ll_{\epsilon, X} 
(B_1 \cdots B_n)^{N+1-d} 
\l( \frac{1}{G(Q)} + E_1 + E_2 \r)
,\] 
where
\begin{align*}
E_1 &= \min \limits_j \{B_j\}^{- \frac{1}{2}} Q^{m(2d-2n+1) +2 + \epsilon}\\
E_2 &=\min \limits_j \{B_j\}^{ - \frac{(N+1)2^{-d}-(d-1)}{4d} + \epsilon}Q^{
2-m(d+2n) + \frac{5md(N+1)}{2^{d-1}(d-1)} + \frac{(N+1)m2^{-d} - m(d-1)}{2d} - \epsilon
}.
\end{align*}

\end{theorem}

Now assuming \ref{sieve theorem}, we'll demonstrate how to establish the main result.
\begin{proof}[Proof of \ref{thin theorem}]
For each $p$ denote by $\overline{A}_p$ the reduction modulo $p$ of the affine cone of $A$ (as in the proof of lemma \ref{P2 count}) and by $\overline{A}$ the collection of all these reductions. 
We start by breaking into dyadic intervals
\[
\#\{ x \in A: H(x) \leq B\}
\leq 
\sum_{\substack{B_1 \cdots B_n \leq B^{\frac{1}{N+1-d}}\\ \text{dyadic}}}
\mathbf{N}(X^n, \mathbf{B}, \overline{A}).
\]
It suffices to prove the estimate in \ref{thin theorem} when $A$ is either a type {\it I} or type {\it II} thin set. This will follow from the $m=1$ case of \ref{sieve theorem}.
If $A$ is a type {\it I} thin set then there is some proper, Zariski closed subset $Z \subset X^n$ which describes it. By \cite[lemma 3.8]{BL} for all primes $p$, we have 
$\# Z(\FF_p) \ll_{Z} p^{n(N+1)-1}.$
It follows that there exists a constant $c>0$ such that
$\omega_p \geq 1 - \frac{c}{p},$
and thus
$\frac{\omega_p}{1- \omega_p} \geq \frac{p}{c} -1.$
This means
\[
G(Q) \geq \sum_{q \leq Q} \mu^2(q)q \prod_{p \mid q} \l( \frac{1}{c} - \frac{1}{p}\r)\gg_{\epsilon, Z} Q^{2 - \epsilon}.
\]

Similarly, if $A$ is a type {\it II} thin set then \cite[lemma 3.8]{BL} implies there is a positive density set of primes $\mathcal{P}$ and a constant $\eta >0$ such that 
$
\frac{\omega_p}{1- \omega_p} \geq \eta,
$ for large enough $p \in \mathcal{P}$.
It follows, as in the proof of \ref{P2 count}, that there exists $\gamma \in [0,1)$ such that
$$
G(Q) \gg_{\epsilon, \mathcal{P}} Q (\log Q)^{- \gamma}
.$$
In either case, \ref{sieve theorem} (with $m=1$) implies that
\[
\mathbf{N}(X^n, \mathbf{B}, \overline{A})
\ll_{A, X}
(B_1 \cdots B_n)^{N+1-d} \l(  Q^{-1}(\log Q)^{\gamma} + E_1 + E_2    \r).
\]
Now setting $Q=\min \limits_j \{B_j\}^{\delta}$ for $\delta >0$ sufficiently small gives the bound
\[
\#\{x \in A : H(x) \leq B\}
\ll_{A,X}
(\log B)^{\gamma}\!\!
\sum_{\substack{B_1 \cdots B_n \leq B^{\frac{1}{N+1-d}}\\ \text{dyadic}}}
(B_1 \cdots B_n)^{N+1-d} \min \limits_j B_j^{-\delta},
\] from which the result follows.
\end{proof}

The rest of this section is dedicated to the proof of \ref{sieve theorem}. 
We count points $x \in X^n(\QQ)$ via their representatives $\x = (\x^{(1)}, \dots, \x^{(n)})\in \ZZ^{n(N+1)}$ where the ${\x^{(j)}}$ are primitive vectors with $F(\x^{(j)}) =0$. Passing to the affine cone, we see that we may bound $\mathbf{N}(X^n,\mathbf{B}, \Omega)$ by
\[
\#\{
\x \in \l(\ZZ_{\text{prim}}^{N+1}\r)^n: B_j/2<\vert\vert \x^{(j)} \vert\vert \leq B_j, F(\x) =0 \text{ and } \x \Mod p^m \in \Omega_{p^m} \forall p
\}.
\]
This quantity can be bounded above using the Selberg sieve.
Let
\[
P(Q) = \prod \limits_{\substack{
p < Q\\
\omega_p >0\\
p \nmid \text{disc}(F)
}}p
\ \ \text{   and   } \ \
\Lambda(\x) =
\prod \limits_{\substack{
p \mid P(Q)\\
\x \Mod p^m  \in \Omega^C_{p^m}
}}p
,\]
where
$
{\Omega_{p^m}^C = \widehat{X^n} \left( \ZZ/p^m \ZZ \right) \setminus \Omega_{p^m}}.
$

Define a sequence $\mathcal{A} = (a_{\lambda})$ of non-negative numbers, supported on finitely many integers $\lambda$, by
\[
a_{\lambda} =\sum_{\substack{
\x \in (\ZZ_{\text{prim}}^{(N+1)})^n\\
F(\x^{(j)}) = 0\\
\Lambda(\x) = \lambda
}}  \prod_{j=1}^n \prod_{i=0}^N W\l( \frac{x_i^{(j)}}{B_j}\r),
\]
for $W$ some appropriate smooth, compactly supported weight function.
Then,
\[
\sum_{(\lambda,P(Q))=1} a_{\lambda} = \!\!\!\!\!
 \sum_{\substack{
\x \in (\ZZ_{\text{prim}}^{(N+1)})^n\\
F(\x^{(j)})=0\\
(\Lambda(\x), P(Q))=1
}}  \prod_{j=1}^n \prod_{i=0}^N W\l( \frac{x_i^{(j)}}{B_j}\r) 
= \!\!\!\!\!
\sum_{\substack{
\x \in (\ZZ_{\text{prim}}^{(N+1)})^n\\
F(\x^{(j)})=0\\
\x \Mod p^m \in \Omega_{p^m} \forall p \mid P(Q)
}}\!\!\!\!  \prod_{j=1}^n \prod_{i=0}^N W\l( \frac{x_i^{(j)}}{B_j}\r)\!\!
.\]
\Cref{sieve theorem} will follow from a suitable upper bound for the $a_{\lambda}$ sum. This is achieved by an appeal to Selberg's upper bound sieve as expressed in \cite[theorem 7.1]{Opera}. In order to apply this, we need an expression of the form
\[
\sum_{\lambda \equiv 0 \Mod q} a_{\lambda} = g(q) Y + r_k(\mathcal{A}),
\] for a constant $Y$ and suitable multiplicative function $g$ and small remainder term $r_q(\mathcal{A})$. This information will be provided by the following result of Schindler--Sofos \cite[lemma 2.1]{SS}.
\begin{lemma}\label{lopsided birch}
Let $g \in \ZZ[x_0, \dots, x_N]$ a polynomial of degree $d \geq 2$. Fix $R >0$ and $\z \in \ZZ^{N+1}$. If ${N> 2^d (d-1)}$, then one has
\begin{align*}
\sum_{\substack{
\x \in \ZZ^{N\!+1}\\
g(\x) =0
}} \prod_{i=0}^N\! W\!\l( \frac{x_i}{R} - z_i\r) \!
\!-\!
\mathfrak{S} \mathfrak{J}_{W}\!
\ll_{\epsilon} \!
 R^{N-d + \frac{1}{2}}+\!
\norm{g}^{ \frac{5(N+1)}{2^{d}(d-1)}  -\frac{ 3 }{ 2 }  } 
R^{N\!+1 -d + \epsilon - \frac{(N\!+1)2^{-d}-(d-1)}{4d}}\!
,\end{align*}
where
\begin{align*}
\mathfrak{J}_{W} &= \int_{-\infty}^{\infty} \int_{\RR^{N+1}} e(\gamma g(\u)) \prod_{i=0}^{N}W
\l(\frac{u_i}{R} \r) \mathrm{d}\u \mathrm{d}\gamma\\
\mathfrak{S} &= \prod_p \sigma_p(g).
\end{align*}
Here $\sigma_p$ is the local density defined as
\[
\sigma_p(g) := \lim_{\ell \rightarrow \infty} p^{-\ell N} \#\{
\x \Mod p^{\ell} : g(\x) \equiv 0 \Mod p^{\ell}
\}
.\]

\end{lemma}
Let $M=q^m$ and $\Omega_M = \prod_{p^m \mid \mid M} \Omega_{p^m}.$
Then for $q \mid P(Q)$, we have
\begin{align*}
\sum_{\lambda \equiv 0 \Mod q} a_{\lambda} =& 
\sum_{\substack{
\x \in (\ZZ_{\text{prim}}^{(N+1)})^n\\
F(\x^{(j)})=0\\
\x \Mod M \in \Omega_M^C
}}  \prod_{j=1}^n \prod_{i=0}^N W\l( \frac{x^{(j)}_i}{B_j}\r) \\
=&
 \sum_{\a \in \Omega_M^C}
\sum_{\substack{
\y \in (\ZZ_{\text{prim}}^{(N+1)})^n\\
F(\a^{(j)} + M\y^{(j)}) =0
}}  \prod_{j=1}^n \prod_{i=0}^N W\l( \frac{a_i^{(j)} + M y_i^{(j)}}{B_j}\r)\\
=&
 \sum_{\a \in \Omega_M^C}
 \prod_{j=1}^n
 \sum_{\substack{
\y^{(j)} \in (\ZZ_{\text{prim}}^{(N+1)})^n\\
F(\a^{(j)} + M\y^{(j)}) =0
}} \prod_{i=0}^N  W\l( \frac{a_i^{(j)} + M y_i^{(j)}}{B_j}\r)
.
\end{align*} 
Now this is in a form where we may apply \ref{lopsided birch}, setting $R = \frac{B_j}{M}$, $ z_i = \frac{a_i^{(j)}}{B_j}$ and $g(\y^{(j)}) = F(\a^{(j)} + M \y^{(j)})$.
Therefore the inner sum over $\y^{(j)}$ can be written as
\begin{align*}
\mathfrak{S} \mathfrak{J}_{W,h_j}
+ O\l(   \l(     \frac{B_j}{M}     \r)^{N-d + \frac{1}{2}}
+
M^{ \frac{5d(N+1)}{2^{d}(d-1)} - \frac{3d}{2} } 
\l(   \frac{B_j}{M}    \r)^{N+1 -d + \epsilon - \frac{(N+1)2^{-d}-(d-1)}{4d}}
\r)
.\end{align*}
Observe that
\begin{align*}
\mathfrak{J}_{W,B_j} 
&= \frac{B_j^{N+1-d}}{M^{N+1}} \int_{-\infty}^{\infty} \int_{\RR^{N+1}} 
e(\beta F(\u))  \prod_{i=0}^N W\l( u_i\r) \mathrm{d}\u \mathrm{d}\beta
=:\frac{B_j^{N+1-d}}{M^{N+1}} \mathfrak{J} .
\end{align*}

The local factors $\sigma_p$ in the singular series are given by
\[
\sigma_p((F(\a^{(j)} + p^m \y^{(j)})) = 
\lim_{\ell \rightarrow \infty} p^{-\ell N} \# \{ \y \Mod p^{\ell}: F(\a^{(j)} +p^m \y^{(j)}) \equiv 0 \Mod p^{\ell} \}.
\]
If $p \nmid M$ then $
\sigma_p(F(\a^{(j)} + p^m\y^{(j)})) = \sigma_p,
$ where $\sigma_p$ is the usual Hardy--Littlewood density associated to $F$.
If $p$ divides $M$, it cannot divide disc$(F)$. Let 
$$\mathbf{N}(\ell) :=  \# \{ \y \Mod p^{\ell}: F(\a^{(j)} +p^m \y) \equiv 0 \Mod p^{\ell} \},$$ it follows via Hensel's lemma that for $\ell > m$ we have
$
 \mathbf{N}(\ell) = p^N  \mathbf{N}(\ell-1),
$ and thus
$
\sigma_p(F(\a^{(j)} + p^m\y)) = p^m.
$
Hence, the singular series factorises as
\[
\mathfrak{S} = \prod_{p \nmid M} \sigma_p \prod_{p^m \mid \mid M} p^m,
\] for any $\a^{(j)}$.
Therefore
\[
\frac{\mathfrak{S}}{M^{N+1}} = \prod_p \sigma_p \prod_{p^m \mid \mid M} \frac{1}{p^{mN}\sigma_p}.
\]
It follows from Hensel's lemma (as above) that $\#{\widehat{X}(\ZZ/p^{\ell}\ZZ) = p^N \#\widehat{X}(\ZZ/p^{\ell-1}\ZZ)} $ for any $\ell>1$. Using this and Deligne's bound, we conclude
\[
\frac{\mathfrak{S}}{M^{N+1}} = \prod_p \sigma_p \!\prod_{p^m \mid \mid M} \frac{1}{\#\widehat{X}\l( \ZZ/p^m\ZZ\r)\l(1 + O(p^{-N/2+1/2}\r))} = c_1 \!\prod_{p^m \mid \mid M} \frac{1}{\#\widehat{X}\l( \ZZ/p^m\ZZ\r)},
\] for some absolute constant $c_1$. Taking the product over all $j$ we get a  main term of size
\[
\#\Omega_M^C c_1 ^n \mathfrak{J}^n \prod_{p^m \mid \mid M} \frac{1}{\#\widehat{X^n}\l( \ZZ/p^m\ZZ\r)}
(B_1 \cdots B_n)^{N+1-d}
.\]
Therefore, there exists a constant $c_2$ (depending at most on $W$ and $F$) such that
\[
\sum_{\lambda \equiv 0 \Mod q} a_{\lambda} =  c_2g(q)(B_1 \cdots B_n)^{N+1-d} + 
O\l(r_q(\mathcal{A})\r)
,\]
 where
\[
g(q) =  \prod_{p \mid q} 
\frac{\#\Omega_{p^m}^C}{\#\widehat{X^n}\l(   
\ZZ/p^m\ZZ
 \r)}
= \prod_{p \mid q} \omega_p.
\]
The remainder term $r_q(\mathcal{A})$ is given by $\#\Omega_{M}^C \left( \frac{B_1 \cdots B_n}{M^n}\right)^{N+1-d}$ multiplied by
\[
\l(\frac{\min \limits_j \{B_j\}}{M}\r)^{\!\! - \frac{1}{2}}\!
M^{-d(n-1)}
+M^{ \frac{5d(N\!+1)}{2^{d}(d-1)} - \frac{3d}{2} - d(n-1) } \!\!
\l(   \frac{\min \limits_j \{B_j\}}{M}    \r)^{ - \frac{(N\!+1)2^{-d}-(d-1)}{4d} + \epsilon}
\!\!.
\]

We estimate $\#\Omega_M^C$ using the following simple bound \[ {\#\Omega_M^C \ll \prod_{p^m \mid \mid M} \#\widehat{X^n}(\ZZ/p^m\ZZ)\ll M^{nN}.}\]
It just remains to compute the error terms
\begin{align*}
\frac{(B_1 \cdots B_n)^{N+1-d}}{\min \limits_j \{B_j\}^{1/2}}
\sum_{q \leq Q^2}\! \tau_3(M) &
q^{m(d+\frac{1}{2}-n(N+1))} \#\Omega_{q^m}^C\\
&\ll_{\epsilon} 
\frac{(B_1 \cdots B_n)^{N+1-d}}{\min \limits_j \{B_j\}^{1/2}} \!\!
\sum_{q \leq Q^2} \!\! q^{m(d+ \frac{1}{2} -n) + \frac{\epsilon}{2}}\\
&\ll_{\epsilon} \frac{(B_1 \cdots B_n)^{N+1-d}}{\min \limits_j \{B_j\}^{1/2}} Q^{m(2d-2n +1)+2 + \epsilon}
\end{align*}
and
\begin{align*}
&\frac{(B_1 \cdots B_n)^{N+1-d}}{\min \limits_j \{B_j\}^{\frac{(N+1)2^{-d}-(d-1)}{4d}-\epsilon}}
\sum_{q \leq Q^2}\! \tau_3(M)  q^{m\l(\frac{5d(N+1)}{2^{d}(d-1)} -  
 \frac{d}{2} -n(N +1) + \frac{(N+1)2^{-d}-(d-1)}{4d} - \epsilon \r)  } \#\Omega_{q^m}^C\\
&\quad{}\ll_{\epsilon}
\frac{(B_1 \cdots B_n)^{N+1-d}}{\min \limits_j \{B_j\}^{\frac{(N+1)2^{-d}-(d-1)}{4d}-\epsilon}}
Q^{
2-m(d+2n) + \frac{5md(N+1)}{2^{d-1}(d-1)} + \frac{(N+1)m2^{-d} - m(d-1)}{2d} - \epsilon
}.
\end{align*}
This finishes the proof of \ref{sieve theorem}.

\def\cprime{$'$}


\begin{thebibliography}{HRS04}

\bibitem[Bir62]{Birch}
Bryan J. Birch. Forms in many variables. {\it Proc. Roy. Soc. Ser. A} 265:245--263, 1962.

\bibitem[BK13]{BeheshtiKumar}
Roya Beheshti and N. Mohan Kumar.
\newblock Spaces of rational curves on complete intersections. \newblock {\em Compos. Math.} 149(6):1041--1060, 2013.

\bibitem[BL17]{BL}
Tim Browning and Daniel Loughran.  Sieving rational points on varieties. {\it Trans. Amer. Math. Soc.} 371 (8):5757--5785, 2019.

\bibitem[BS18]{Browning2018Free-rational-c}
Tim Browning and Will Sawin. Free rational curves on low degree hypersurfaces and the circle method. (arXiv:1810.06882).

\bibitem[BV17]{BrowningVishe}
Tim Browning and Pankaj Vishe.
\newblock Rational curves on smooth hypersurfaces of low degree.
\newblock {\em Algebra Number Theory} 11(7):1657--1675, 2017.

\bibitem[CR19]{CoskunRiedl}
Izzet Coskun and Eric Riedl. 
\newblock  Normal bundles of rational curves on complete intersections.
\newblock { \em Commun.Cont. Math.}, Vol. 21, No. 02, 1850011, 2019. 

\bibitem[Deb01]{DebarreHDAG}
Olivier Debarre. 
\newblock {\em Higher-dimensional algebraic geometry.}
\newblock Universitext. Springer-Verlag, New York, 2001. xiv+233 pp.

\bibitem[FMT89]{FMT}
Jens Franke, Yuri Manin and Yuri Tschinkel.
\newblock Rational points of bounded height on Fano varieties.
\newblock \textit{Invent. Math.} \textbf{95}, 421--435, 1989.

\bibitem[FI10]{Opera}
John Friedlander and Henryk Iwaniec. {\em Opera De Cribro}. { Amer. Math. Soc.}, 2010.


\bibitem[HRS04]{Harris2004Rational-curves}
Joe Harris, Mike Roth, and Jason Starr.
\newblock Rational curves on hypersurfaces of low degree.
\newblock {\em J. reine angew. Math.}, 571:73--106, 2004.

\bibitem[KM94]{Kontsevich1994Gromov-Witten-c}
Maxim Kontsevich and Yuri Manin.
\newblock Gromov-{W}itten classes, quantum cohomology, and enumerative
  geometry.
\newblock {\em Comm. Math. Phys.}, 164(3):525--562, 1994.

\bibitem[Kow08]{Kow}
Emmanuel Kowalski.
\newblock {\em The large sieve and its applications.}
\newblock Cambridge Tracts in Mathematics \textbf{175}, Cambridge University Press, Cambridge, 2008.


\bibitem[Ran89]{Ran1989Enumerative-geo}
Ziv Ran.
\newblock Enumerative geometry of singular plane curves.
\newblock {\em Invent. Math.}, 97(3):447--465, 1989.

\bibitem[RY16]{RiedlYang}
Eric Riedl and David Yang.
\newblock Kontsevich spaces of rational curves on Fano hypersurfaces.
\newblock {\em J. reine angew. Math.} 748:207--225, 2019.

\bibitem[SS19]{SS}
Damaris Schindler and Efthymios Sofos. 
\newblock Sarnak's saturation problem for complete intersections. 
\newblock {\it Mathematika} 65(1):1--56, 2019.

\bibitem[Ser16]{Serre2016Topics-in-Galoi}
Jean-Pierre Serre.
\newblock {\em Topics in Galois theory}.
\newblock AK Peters/CRC Press, 2016.

\bibitem[She12]{Shen_normal_bundle}
Mingmin Shen.
\newblock { On the normal bundles of rational curves on Fano 3-folds}. 
\newblock {\em Asian J. Math.} Vol. 16, No. 2, 237--270, 2012.

\bibitem[Sk97]{Skinner}
Chris Skinner.
\newblock Forms over number fields and weak approximation.
\newblock {\em Compositio Math.} 106(1):11--29, 1997. 

\bibitem[Vai95]{Vainsencher1995Enumeration-of-}
Israel Vainsencher.
\newblock Enumeration of $ n $-fold tangent hyperplanes to a surface.
\newblock {\em J. Alg. Geom.}, 4:503--526, 1995.

\bibitem[VAZ09]{Varilly-Alvarado-Zywina}
Tony V\'arilly-Alvarado and David Zywina. 
\newblock Arithmetic $E_8$ lattices with maximal Galois action. 
\newblock {\em LMS J. Comput. Math.} 12:144--165, 2009. 

\bibitem[Wir67]{Wirsing}
Eduard Wirsing. 
\newblock Das asymptotische Verhalten von Summen {\"u}ber multiplikative Funktionen. II.
\newblock \textit{Acta Math. Acad. Sci. Hungar.} 18:411--467, 1967.


\bibitem[Zeu73]{ZeuthenAlmindelige-Ege}
Hieronymus Georg~Zeuthen.
\newblock Almindelige Egenskaber ved systemer af plane Kurver
\newblock {\em Kongelige Danske Videnskabernes Selskabs skrifter, Naturvidenskabelig og Mathematisk}, Afd. 10, Bd. IV,  (1873) 285--393.
%
%
\end{thebibliography}
\end{document}